%% file: Main.tex
\documentclass[12pt,twoside]{article}

	\usepackage{fancyhdr} 
	\usepackage[margin=1in]{geometry}  
	\usepackage{graphicx}              
	\usepackage{amsmath}               
	\usepackage{amsfonts}              
	\usepackage{amsthm}                
	\usepackage{enumitem}
	\usepackage{amssymb}
	\usepackage{mathtools}
	\usepackage{hyperref}
	\ExplSyntaxOn
	\cs_new_eq:NN \intfromalph \int_from_alph:n
	\ExplSyntaxOff

	\newtheorem{thm}{Theorem}[section]
	\newtheorem{lem}[thm]{Lemma}
	\newtheorem{prop}[thm]{Proposition}

	\theoremstyle{remark}

	\theoremstyle{definition}
	\newtheorem{exa}[thm]{Example}
	\newtheorem{Def}[thm]{Definition}
	\theoremstyle{plain}

\begin{document}

		\nocite{*}
		
		\title{The Multiplicity One Defect and Congruence–Modular Exponent Ratios}
		
		\author{Mohammad Masih Hamidi\thanks{E-mail: mohammad.hamidi@mail.mcgill.ca}\\ 
			Department of Mathematics \\
			McGill University\footnote{Address: 805 Sherbrooke Street West, Montreal, QC, Canada}\\
		}
		\date{}
		\maketitle
		
		\input{Abstract}
		\input{Introduction}
	    \input{Defect}
        \input{TheoremB}

	    \input{Gorenstein}
	    \bibliographystyle{plain}
	    \bibliography{Refe}	
	    
	\end{document}

%% file: Abstract.tex
\begin{abstract}
	We introduce a multiplicity one defect that quantitatively measures the
	failure of multiplicity one for the Jacobian of the Shimura curve
	\(X^D(M)\) attached to an indefinite rational quaternion algebra of
	discriminant \(D\) and level \(M\), localized at a maximal ideal of the
	Hecke algebra.  By constructing a certain lattice inside the local Tate module, we obtain an explicit upper bound for the defect in terms of congruences between eigenforms and the lattice–Tate module discrepancies associated with the individual newforms. We then show that the
	local ratio between the congruence exponent and the modular exponent of
	modular abelian varieties attached to newforms is bounded above by the
	multiplicity one defect. 
	Finally, we examine the relation with the Gorenstein property
	and give an explicit example at level \(126\) where the local
	Hecke algebra is Gorenstein even though multiplicity one fails.
	In this example, we compute the multiplicity one defect to be
	\(3\), which is also the local congruence--modular exponent ratio.

\end{abstract}

%% file: Introduction.tex
\section{Introduction}

Throughout this paper, let \(N=DM\), where \(D\) is a product of an even
number of distinct primes and \((D,M)=1\). We denote by \(X^D(M)\) the Shimura
curve over \(\mathbf Q\) attached to the indefinite rational quaternion algebra
of discriminant \(D\) and level \(M\), and by \(J^D(M)\) its Jacobian. Let
$
f\in S_2^{\mathrm{new}}(\Gamma_0^D(M))
$
be a quaternionic newform of weight \(2\) and level \(M\), and let \(A_f\)
denote the dual of the modular abelian variety associated with \(f\). The
variety \(A_f\) occurs as a Hecke-stable abelian subvariety of \(J^D(M)\).

Let \(\mathbb T^D(M)\) denote the Hecke algebra acting on \(J^D(M)\),
generated by the operators \(T_p\) for primes \(p\nmid M\) and \(U_p\) for
primes \(p\mid M\). Associated with \(A_f\) are two closely related arithmetic
invariants: the modular exponent and the congruence exponent. In the notation
of Section~3.2,
\[
\delta^D(M)
=
\exp\!\left(\mathbb T^D(M)_{A_f}/S^D(M)\right),
\qquad
r^D(M)
=
\exp\!\left(\mathbb T^D(M)_{A_f}/R^D(M)\right),
\]
where \(S^D(M)\) is the intersection ideal and \(R^D(M)\) is the congruence
ideal. Their relation is closely connected with multiplicity one for the
action of the Hecke algebra on the Jacobian.

Let \(\mathfrak m\) be a maximal ideal of \(\mathbb T^D(M)\) of residue
characteristic \(\ell\). Under the multiplicity one hypothesis
$
\dim_{\mathbb T^D(M)/\mathfrak m}J^D(M)[\mathfrak m]=2,
$
the localized intersection and congruence ideals coincide, and hence the local
modular and congruence exponents agree. Without multiplicity one, however, the
two exponents need not be equal. Since the congruence ideal is contained in
the intersection ideal, the modular exponent divides the congruence exponent.
It is therefore natural to ask whether the discrepancy between them can be
measured in terms of the failure of multiplicity one. The purpose of this
paper is to give such a quantitative measure.
Write
$
\widehat{\mathbb T^D(M)_{\mathfrak m}}
\cong
(\mathbb T^D(M)\otimes_{\mathbb Z}\mathbb Z_\ell)_{\mathfrak m}
$
for the completed local Hecke algebra, and put
$
L=T_{\mathfrak m}(J^D(M)),
$
where \(T_{\mathfrak m}(J^D(M))\) denotes the
covariant \(\mathfrak m\)-adic Tate module, naturally
a finite \(\widehat{\mathbb T^D(M)_{\mathfrak m}}\)-module.
We define the \emph{multiplicity one defect}
at \(\mathfrak m\) by
\[
\delta_{\mathrm{def}}(\mathfrak m)
=
\min_V \exp(L/V),
\]
where \(V\) ranges over all free \(\widehat{\mathbb T^D(M)_{\mathfrak m}}\)-submodules
of \(L\) of rank \(2\) and finite index. Thus
\(\delta_{\mathrm{def}}(\mathfrak m)\) measures how
far the integral Tate module is from being free
of rank \(2\) over the completed local Hecke
algebra.
 The existence of such lattices, and hence the well-definedness of
\(\delta_{\mathrm{def}}(\mathfrak m)\), will be established below. In the
multiplicity one case, \(L\) is free of rank \(2\) over \(\widehat{\mathbb T^D(M)_{\mathfrak m}}\), so one may take
\(V=L\) and obtain
\[
\delta_{\mathrm{def}}(\mathfrak m)=1.
\]
Our main result shows that this defect directly controls the discrepancy
between the local congruence and modular exponents. Let
\(r_{\mathfrak m}^{D}(M)\) and \(\delta_{\mathfrak m}^{D}(M)\) denote the
local congruence and modular exponents, respectively.

\medskip

\noindent\textbf{Theorem A.}
For every maximal ideal \(\mathfrak m\) of \(\mathbb T^D(M)\),
\[
\frac{r_{\mathfrak m}^{D}(M)}
{\delta_{\mathfrak m}^{D}(M)}
\leq
\delta_{\mathrm{def}}(\mathfrak m).
\]

\medskip
Thus the failure of equality between the local congruence and modular
exponents is controlled by the failure of the localized Tate module to be a
free rank-\(2\) module over the localized Hecke algebra. When multiplicity one
holds, \(\delta_{\mathrm{def}}(\mathfrak m)=1\), and Theorem~A recovers
\[
r_{\mathfrak m}^{D}(M)=\delta_{\mathfrak m}^{D}(M).
\]

We next address the existence of the lattices appearing in the definition of
\(\delta_{\mathrm{def}}(\mathfrak m)\). In fact, we prove more: we construct
such a lattice explicitly and obtain a quantitative upper bound for the
multiplicity one defect. It is worth remarking that the existence of such a
lattice already follows from the characteristic-zero rank-two freeness
\[
L\otimes_{\mathbb Z_\ell}\mathbb Q_\ell
\cong
\left(\widehat{\mathbb T^D(M)_{\mathfrak m}}\otimes_{\mathbb Z_\ell}\mathbb Q_\ell\right)^{\oplus 2}
\]
(see, for instance, Section~4 of \cite{helm2007maps}). Our construction,
however, is explicit and provides an upper bound for the defect.

For each \(g\), put
$
\mathcal A_{g,\mathfrak m}
=
\mathcal A_g\otimes_{\mathbb T^D(M)}\widehat{\mathbb T^D(M)_{\mathfrak m}}.
$
Let
\[
A_{\mathfrak m}
=
\prod_{I_g\subseteq\mathfrak m}
\mathcal A_{g,\mathfrak m},
\qquad
c_{\mathfrak m}
=
\exp(A_{\mathfrak m}/\widehat{\mathbb T^D(M)_{\mathfrak m}}),
\]
where \(\mathcal A_g\) is the image of the
full level-\(M\) Hecke algebra on the
\(g\)-old eigenspace and
$
I_g=\ker(\mathbb T^D(M)\to\mathcal A_g).
$
For each \(g\) with \(I_g\subseteq\mathfrak m\), put
\[
T_g=\widehat{\mathbb T^D(M)_{\mathfrak m}}/I_g\widehat{\mathbb T^D(M)_{\mathfrak m}}
\cong
(\mathbb T^D(M)/I_g)
\otimes_{\mathbb T^D(M)}\widehat{\mathbb T^D(M)_{\mathfrak m}}.
\]
Let
$
L_g=L[I_g\widehat{\mathbb T^D(M)_{\mathfrak m}}]
$
be the corresponding integral Tate-module component.
We construct a free \(T_g\)-submodule
$
V_g\subseteq L_g
$
of rank \(2\) and finite index, and set
\[
e_{\mathfrak m}
=
\operatorname{lcm}_{I_g\subseteq\mathfrak m}
\exp(L_g/V_g).
\]

\medskip

\noindent\textbf{Theorem B.}
There exists a free \(\widehat{\mathbb T^D(M)_{\mathfrak m}}\)-submodule
$
V\subset T_{\mathfrak m}(J^D(M))
$
of rank \(2\) such that the exponent of the abelian group
$
T_{\mathfrak m}(J^D(M))/V
$
is at most \(c_{\mathfrak m}^{\,2}e_{\mathfrak m}\). Consequently,
\[
\delta_{\mathrm{def}}(\mathfrak m)
\leq
c_{\mathfrak m}^{\,2}e_{\mathfrak m}.
\]

\medskip

Theorem~B shows, in particular, that the collection of lattices
occurring in the definition of \(\delta_{\mathrm{def}}(\mathfrak m)\)
is nonempty. Its proof uses the characteristic-zero decomposition
of the local Hecke action into newform and oldform components.
The construction treats new and old components uniformly
by using the characteristic-zero rank-two freeness of the
Tate module over the corresponding Hecke algebra on each
component, as used in the proof of  Lemma~\ref{lem:free-hecke-factor-lattices}.

The bound in Theorem~B separates two sources of failure of integral
multiplicity one: the congruence gluing between the different
characteristic-zero Hecke eigenspaces, measured by \(c_{\mathfrak m}\), and
the discrepancy between the integral Tate-module lattice and the free lattice
inside each individual eigenspace, measured by \(e_{\mathfrak m}\).

To prove Theorem~A, let \(V\subset L\) be any free
\(\mathbb T^D(M)_{\mathfrak m}\)-submodule of rank \(2\) and finite index.
Associated with \(V\) is an abelian variety \(J^D(M)_V\),
Hecke-equivariantly isogenous to \(J^D(M)\), for which
\[
T_{\mathfrak m}(J^D(M)_V)
\simeq
V
\simeq
\widehat{\mathbb T^D(M)_{\mathfrak m}}^{\,2}.
\]
Multiplicity one therefore holds integrally on the modified Tate module. The
congruence and modular exponents agree for the corresponding factor of
\(J^D(M)_V\), and comparison through the isogeny gives
\[
\frac{r_{\mathfrak m}^{D}(M)}
{\delta_{\mathfrak m}^{D}(M)}
\leq
\exp(L/V).
\]
Since this holds for every such lattice \(V\), minimizing over \(V\) gives
Theorem~A. We also give a second proof in terms of the saturation of the Hecke
algebra in the endomorphism algebra.

It is natural to compare multiplicity one with another important local
property of Hecke algebras, namely Gorensteinness. These properties are often
closely related in arithmetic applications, and failure of the Gorenstein
property is frequently accompanied by failure of multiplicity one.
Nevertheless, the converse implication does not hold in general. In
Section~4 we give an explicit example at level \(126\).

Let \(E=126\mathrm{A}1\) in Cremona notation, equivalently the optimal curve
\(126.b5\) in the LMFDB, and let \(\mathfrak m\) be the maximal ideal of
residue characteristic \(3\) associated with its newform. The modular degree
of \(E\) is \(8\), whereas its congruence number is \(24\); hence multiplicity
one fails at \(\mathfrak m\). Writing \(\mathbb T=\mathbb T^1(126)\), we
nevertheless show that the \(\mathfrak m\)-adic completion of the local Hecke
algebra is
\[
\widehat{\mathbb T^D(M)_{\mathfrak m}}
\simeq
\mathbf Z_3\times_{\mathbf F_3}\mathbf Z_3[\sqrt 3]
\simeq
\mathbf Z_3[X]/(X^3-3X).
\]
This is a complete intersection and therefore Gorenstein. Since a Noetherian
local ring is Gorenstein if and only if its completion is Gorenstein, it
follows that \(\mathbb T_{\mathfrak m}\) is Gorenstein as well. Thus
Gorenstein does not imply multiplicity one. We also compute for this example that $\delta_{\mathrm{def}}(\mathfrak{m})=3$.

The paper is organized as follows. In Section~2 we construct the free Hecke
lattices appearing in the definition of the multiplicity one defect, prove
Theorem~B, and obtain the explicit bound
\[
\delta_{\mathrm{def}}(\mathfrak m)
\leq
c_{\mathfrak m}^{\,2}e_{\mathfrak m}.
\]
In Section~3 we associate abelian varieties to free Hecke lattices, compare
their congruence and modular exponents, and prove the main result,
Theorem~A, by two different arguments. Finally, in Section~4 we study the
relation between multiplicity one and Gorensteinness and give the explicit
level-\(126\) example described above.

%% file: Defect.tex
\section{Multiplicity One Defect}\label{sec:multiplicity-defect}

In this section we construct free Hecke lattices inside the local Tate module
and prove Theorem~B. The construction uses the characteristic-zero rank-two
property recalled in the introduction; in particular, it does not require an
integral decomposition into degeneracy copies of an oldform.

Let \(\mathbb T=\mathbb T^D(M)\), let \(\mathfrak m\) be a maximal ideal of
\(\mathbb T\) of residue characteristic \(\ell\), and put
\[
T=\widehat{\mathbb T_{\mathfrak m}}
\cong (\mathbb T\otimes_{\mathbb Z}\mathbb Z_\ell)_{\mathfrak m},
\qquad
L=T_{\mathfrak m}(J^D(M)).
\]
Here \(T_{\mathfrak m}(J^D(M))\) is the covariant
\(\mathfrak m\)-adic Tate module, and \(L\) is a finite, torsion-free
\(\mathbb Z_\ell\)-module with a natural \(T\)-action.
We define the \emph{multiplicity one defect} at \(\mathfrak m\) by
\[
\delta_{\mathrm{def}}(\mathfrak m)
=\min_{\substack{V\subseteq L\\V\simeq T^2\\\lvert L/V\rvert<\infty}}
\exp(L/V).
\]
The proof below shows that the family over which the minimum is taken is
nonempty.

\subsection{Lattices in Hecke components}

Let \(g\) run over representatives of the primitive Hecke eigensystems of
exact levels \(Z_g\mid M\) whose oldspaces occur at level \(M\). Write
\(\mathcal A_g\) for the image of the full level-\(M\) Hecke algebra
\(\mathbb T\) on the entire \(g\)-oldspace, and set
$
I_g=\ker(\mathbb T\longrightarrow\mathcal A_g).
$
Only the factors for which \(I_g\subseteq\mathfrak m\) contribute after
localization. Denote this set of indices by
$
\mathcal G_{\mathfrak m}=
\{g:I_g\subseteq\mathfrak m\},
$
and, for \(g\in\mathcal G_{\mathfrak m}\), put
\[
T_g=T/I_gT
\cong\widehat{(\mathcal A_g)_{\mathfrak m}},
\qquad
L_g=L[I_gT]
=\{x\in L:(I_gT)x=0\}.
\]
We use the characteristic-zero decomposition of the Hecke action into its
primitive oldspace factors and the rational rank-two property
\begin{equation}\label{eq:rational-hecke-rank-two}
	T\otimes_{\mathbb Z_\ell}\mathbb Q_\ell
	\;\xrightarrow{\ \sim\ }
	\prod_{g\in\mathcal G_{\mathfrak m}}
	(T_g\otimes_{\mathbb Z_\ell}\mathbb Q_\ell),
	\qquad
	L\otimes_{\mathbb Z_\ell}\mathbb Q_\ell
	\;\simeq\;
	(T\otimes_{\mathbb Z_\ell}\mathbb Q_\ell)^2.
\end{equation}
The first isomorphism in \eqref{eq:rational-hecke-rank-two} gives a natural
injection of \(\mathbb Z_\ell\)-orders
\begin{equation}\label{eq:hecke-product-inclusion}
	T\hookrightarrow A_{\mathfrak m},
	\qquad
	A_{\mathfrak m}
	:=\prod_{g\in\mathcal G_{\mathfrak m}}T_g,
\end{equation}
whose cokernel is finite. Define
$
c_{\mathfrak m}=\exp(A_{\mathfrak m}/T).
$
Thus \(c_{\mathfrak m}A_{\mathfrak m}\subseteq T\).

\begin{lem}\label{lem:free-hecke-factor-lattices}
	For every \(g\in\mathcal G_{\mathfrak m}\), there exist
	\(x_g,y_g\in L_g\) such that
	\[
	V_g=T_gx_g+T_gy_g\subseteq L_g
	\]
	is free of rank \(2\) over \(T_g\) and \(L_g/V_g\) is finite.
\end{lem}

\begin{proof}
	Under the decomposition \eqref{eq:rational-hecke-rank-two}, the
	\(I_gT\)-annihilated rational submodule is precisely the \(g\)-factor.
	Since \(L\) is \(\mathbb Z_\ell\)-torsion-free, clearing denominators
	shows that localization commutes with taking this annihilator. Consequently,
	\[
	L_g\otimes_{\mathbb Z_\ell}\mathbb Q_\ell
	\simeq
	\bigl(T_g\otimes_{\mathbb Z_\ell}\mathbb Q_\ell\bigr)^2.
	\]
	Choose \(u_g,v_g\) corresponding to the two standard basis vectors under
	such an isomorphism. For some \(a\geq0\), both
	$
	x_g=\ell^a u_g,
	$and
	$
	y_g=\ell^a v_g
	$
	belong to \(L_g\). The \(T_g\)-linear map
	\[
	T_g^2\longrightarrow L_g,
	\qquad
	(s,t)\longmapsto sx_g+ty_g
	\]
	becomes an isomorphism after tensoring with \(\mathbb Q_\ell\).
	Because \(T_g\) is \(\mathbb Z_\ell\)-torsion-free, the map is
	injective. Its cokernel is a finitely generated torsion
	\(\mathbb Z_\ell\)-module and hence finite. Its image is the required
	lattice \(V_g\).
\end{proof}
Fix the lattices \(V_g\) given by Lemma~\ref{lem:free-hecke-factor-lattices}
and define
\[
e_{\mathfrak m}
=\operatorname{lcm}_{g\in\mathcal G_{\mathfrak m}}
\exp(L_g/V_g).
\]
The integer \(e_{\mathfrak m}\) depends on the chosen factor lattices.
The bound in Theorem~B holds for every such choice.

\subsection{Construction of the global lattice}
\begin{proof}[Proof of Theorem~B, Part~1]
	Choose \(x_g,y_g\) as in Lemma~\ref{lem:free-hecke-factor-lattices}
	and put
	\[
	X=\sum_{g\in\mathcal G_{\mathfrak m}}x_g,
	\qquad
	Y=\sum_{g\in\mathcal G_{\mathfrak m}}y_g,
	\qquad
	V=TX+TY\subseteq L.
	\]
	The rational direct-product decomposition in
	\eqref{eq:rational-hecke-rank-two} implies that the sum of the \(L_g\)
	is direct inside \(L\). Indeed, any relation among elements of these
	submodules becomes zero in each separate rational factor; torsion-freeness
	then gives the integral assertion.
	Suppose that \(tX+sY=0\) for \(t,s\in T\). Its \(g\)-component is
	\[
	\overline t_gx_g+\overline s_gy_g=0,
	\qquad
	\overline t_g,\overline s_g\in T_g.
	\]
	Since \(x_g,y_g\) freely generate \(V_g\), we obtain
	\(\overline t_g=\overline s_g=0\) for every \(g\). The injection
	\eqref{eq:hecke-product-inclusion} now gives \(t=s=0\). Hence
	\(V\simeq T^2\).
	After tensoring with \(\mathbb Q_\ell\), the central idempotents of
	\(T\otimes\mathbb Q_\ell\) select the \(g\)-components of \(X\)
	and \(Y\). Lemma~\ref{lem:free-hecke-factor-lattices} therefore gives
	\[
	V\otimes_{\mathbb Z_\ell}\mathbb Q_\ell
	=L\otimes_{\mathbb Z_\ell}\mathbb Q_\ell.
	\]
	Thus \(L/V\) is finite, proving the first part of Theorem~B.
\end{proof}

\subsection{The congruence bound}
\begin{proof}[Proof of Theorem~B, Part~2]
	Set
	\[
	L^{\mathrm{dec}}=
	\bigoplus_{g\in\mathcal G_{\mathfrak m}}L_g,
	\qquad
	V^{\mathrm{dec}}=
	\bigoplus_{g\in\mathcal G_{\mathfrak m}}V_g.
	\]
	These are submodules of \(L\), and
	\(V^{\mathrm{dec}}\subseteq L^{\mathrm{dec}}\subseteq L\).
	
	For \(g\in\mathcal G_{\mathfrak m}\), let \(\varepsilon_g\in A_{\mathfrak m}\)
	be the idempotent that is \(1\) in the \(g\)-factor and \(0\) in
	all other factors. Since
	\(c_{\mathfrak m}A_{\mathfrak m}\subseteq T\), the element
	\(c_{\mathfrak m}\varepsilon_g\) belongs to \(T\).
	For any \(z\in L\),
	$
	c_{\mathfrak m}z
	=\sum_{g\in\mathcal G_{\mathfrak m}}
	(c_{\mathfrak m}\varepsilon_g)z.
	$
	Moreover, each summand lies in \(L_g\): if \(u\in I_gT\), then
	\(u(c_{\mathfrak m}\varepsilon_g)=0\) in \(T\), as can be checked
	in the product ring \(A_{\mathfrak m}\). Hence
	\begin{equation}\label{eq:first-lattice-containment}
		c_{\mathfrak m}L\subseteq L^{\mathrm{dec}}.
	\end{equation}
	By the definition of \(e_{\mathfrak m}\),
	\begin{equation}\label{eq:second-lattice-containment}
		e_{\mathfrak m}L^{\mathrm{dec}}\subseteq V^{\mathrm{dec}}.
	\end{equation}
	Because each \(V_g\simeq T_g^2\), the chosen generators induce
	an isomorphism
	\[
	A_{\mathfrak m}^2\xrightarrow{\ \sim\ }V^{\mathrm{dec}},
	\qquad
	\bigl((a_g)_g,(b_g)_g\bigr)
	\longmapsto\sum_g(a_gx_g+b_gy_g).
	\]
	Under this isomorphism, the lattice \(V=TX+TY\) corresponds exactly
	to \(T^2\subseteq A_{\mathfrak m}^2\). It follows that
	$
	V^{\mathrm{dec}}/V
	\simeq(A_{\mathfrak m}/T)^2,
	$
	and consequently
	\begin{equation}\label{eq:third-lattice-containment}
		c_{\mathfrak m}V^{\mathrm{dec}}\subseteq V.
	\end{equation}
	Combining \eqref{eq:first-lattice-containment},
	\eqref{eq:second-lattice-containment}, and
	\eqref{eq:third-lattice-containment} gives
	$
	c_{\mathfrak m}^{\,2}e_{\mathfrak m}L\subseteq V.
	$
	Therefore
	$
	\delta_{\mathrm{def}}(\mathfrak m)\leq\exp(L/V)
    \leq c_{\mathfrak m}^{\,2}e_{\mathfrak m},
	$
	as required.
\end{proof}

%% file: TheoremB.tex
\section{Exponents via Isogeny}

In this section, we prove Theorem~A. We first prepare the setup.

Fix a maximal ideal \(\mathfrak m\) of \(\mathbb T^D(M)\) of residue characteristic \(\ell\), and put
\[
T=\widehat{(\mathbb T^D(M))_{\mathfrak m}}
\cong (\mathbb T^D(M)\otimes_{\mathbb Z}\mathbb Z_\ell)_{\mathfrak m}.
\]
For any finite \(\mathbb T^D(M)\)-algebra \(\mathcal E\), its \(\mathfrak m\)-local
factor in this section means \(\mathcal E\otimes_{\mathbb T^D(M)}T\).
For finite-length Hecke modules, ordinary localization and completion agree.

Throughout, for an abelian variety \(J\) defined over a number field $K$, we write
\(\operatorname{End}(J)\) for the ring of endomorphisms of $J_{\bar{\mathbb{Q}}}$. We also write
\(\operatorname{End}(J)_K\) for the subring of endomorphisms defined over
\(K\). If a ring \(\mathbb T\) acts on \(J\), we write
\[
\operatorname{End}_{\mathbb T}(J)
=
\{\phi\in\operatorname{End}(J):\phi t=t\phi
\text{ for all }t\in\mathbb T\}
\]
for the subring of \(\overline{\mathbb Q}\)-endomorphisms commuting with the
\(\mathbb T\)-action.

\subsection{Abelian Varieties Associated to Lattices}

Write \(L=\mathcal T_{\mathfrak m}(J^D(M))\). A \emph{lattice} will mean a
free \(T\)-submodule \(V\subset L\) of rank \(2\)
and finite index. 
Composing the canonical principal polarization of $J^D(M)$
with the Atkin--Lehner involution yields a
Hecke-equivariant identification of its covariant and
contravariant $\mathfrak m$-adic Tate modules.
We therefore regard $L$ as the contravariant
$\mathfrak m$-adic Tate module.

There is a natural construction of an abelian variety \(J^D(M)_V\), equipped
with a faithful action of \(\mathbb T^D(M)\), which is
\(\mathbb T^D(M)\)-equivariantly isogenous to \(J^D(M)\) and whose
contravariant \(\mathfrak m\)-adic Tate module identifies with \(V\).

Let \(\ell\) be the residue characteristic of \(\mathfrak m\). Write
\(J^D(M)[\mathfrak m^\infty]\) for the \(\mathfrak m\)-divisible group of
\(J^D(M)(\overline{\mathbb Q})\) and \(\mu_{\ell^\infty}\) for the
\(\ell\)-divisible group of \(\overline{\mathbb Q}^{\times}\). By definition,
\[
\mathcal T_{\mathfrak m}(J^D(M))
=
\operatorname{Hom}\!\left(
J^D(M)[\mathfrak m^\infty],\mu_{\ell^\infty}
\right),
\]
which gives a natural nondegenerate pairing
\[
\mathcal T_{\mathfrak m}(J^D(M))
\times
J^D(M)[\mathfrak m^\infty]
\longrightarrow
\mu_{\ell^\infty}.
\]
Let \(G\subset J^D(M)[\mathfrak m^\infty]\) be the annihilator of \(V\)
under this pairing. Since \(V\) has finite index in
\(\mathcal T_{\mathfrak m}(J^D(M))\), the subgroup \(G\) is finite and
\(\mathbb T^D(M)\)-stable. Define \(J^D(M)_V:=J^D(M)/G\). The quotient map
$
\pi:J^D(M)\longrightarrow J^D(M)_V
$
is \(\mathbb T^D(M)\)-equivariant, so \(J^D(M)_V\) inherits a natural
\(\mathbb T^D(M)\)-action. Since \(G\) is finite, it is defined over a
finite extension \(K/\mathbb Q\), and hence so is \(J^D(M)_V\). By
construction,
\[
\mathcal T_{\mathfrak m}(J^D(M)_V)
\cong
V
\subset
\mathcal T_{\mathfrak m}(J^D(M)).
\]

\begin{prop}\label{prop:dual-tate-module}
	The covariant \(\mathfrak m\)-adic Tate module of \(J^D(M)_V^{\vee}\),
	denoted by
	\(\widetilde{\mathcal T}_{\mathfrak m}(J^D(M)_V^{\vee})\), is isomorphic to
	$
	\mathcal T_{\mathfrak m}(J^D(M)_V)\cong V
    $
	as a \(T\)-module.
\end{prop}

\begin{proof}
	Let
	\[
	\langle\cdot,\cdot\rangle:
	\mathcal T_\ell(J^D(M)_V)
	\times
	\mathcal T_\ell(J^D(M)_V^{\vee})
	\longrightarrow
	\mathbb Z_\ell(1)
	\]
	be the perfect Weil pairing. For $t \in \mathbb{T}^D(M)$, let $t^{\vee}$ be the endomorphism of $J^D(M)^{\vee}_V$ induced by dualization. One has
    $
	\langle tx,y\rangle
	=
	\langle x,t^{\vee}y\rangle
	$
	for every \(x,y\).
	In the decomposition
	\[
	(\mathbb T^D(M)\otimes_{\mathbb Z}\mathbb Z_\ell)
	=
	\prod_{\mathfrak n\mid\ell}
	\widehat{(\mathbb T^D(M))_{\mathfrak n}},
	\]
	let \(\varepsilon_{\mathfrak m}\) be the idempotent corresponding to
	\(T\). Since
	\(\varepsilon_{\mathfrak n}\varepsilon_{\mathfrak m}=0\) for
	\(\mathfrak n\neq\mathfrak m\), the Weil pairing restricts to a
	perfect pairing on the \(\mathfrak m\)-parts. Equivalently, the covariant
	\(\mathfrak m\)-adic Tate module of the dual is identified with the
	contravariant \(\mathfrak m\)-adic Tate module of \(J^D(M)_V\). Hence
	\[
	\widetilde{\mathcal T}_{\mathfrak m}(J^D(M)_V^{\vee})
	\cong
	\mathcal T_{\mathfrak m}(J^D(M)_V)
	\cong
	V,
	\]
	compatibly with the Hecke action.
\end{proof}

\subsection{Congruence and Modular Exponents}

Let $f\in S_2^{\text{new}}(\Gamma_0^D(M))$ be a quaternionic newform of weight 2 and level $M$. Let \(V\) be a lattice in the sense of Section 3.1, and write \(A=A_f\) for the modular abelian subvariety of $J^D(M)$ attached to $f$. Choose a
\(\mathbb T^D(M)\)-stable abelian subvariety \(B\) complementary to \(A\), so
that
$
J^D(M)=A+B
$
and $A \cap B$ is finite.
The congruence and modular exponents of \(A\) are defined as in
\cite[Section~3]{agashe2011modular}.

\begin{Def}
	Let
	\[
	\pi_A:\mathbb T^D(M)\longrightarrow\operatorname{End}(A),
	\qquad
	\pi_B:\mathbb T^D(M)\longrightarrow\operatorname{End}(B),
	\]
	and write
	\(\mathbb T^D(M)_A=\pi_A(\mathbb T^D(M))\) and
	\(\mathbb T^D(M)_B=\pi_B(\mathbb T^D(M))\).
	
	\begin{enumerate}[label=(\alph*),start=\intfromalph{i}]
		\item
		The congruence ideal attached to \(A\) is
		\(R:=\pi_A(\ker(\pi_B))\), and the congruence exponent is
		\[
		r^D(M)=\exp\!\left(\mathbb T^D(M)_A/R\right).
		\]
		
		\item
		The intersection ideal attached to \(A\) is
		\(S:=\operatorname{Ann}_{\mathbb T^D(M)_A}(A\cap B)\), and the modular
		exponent is
		\[
		\delta^D(M)=\exp\!\left(\mathbb T^D(M)_A/S\right).
		\]
	\end{enumerate}
\end{Def}

Let \(A_V\) and \(B_V\) be the images of \(A\) and \(B\) under
\(\pi:J^D(M)\to J^D(M)_V\). Then
\[
J^D(M)_V=A_V+B_V.
\]
Let \(I=\ker(\pi_A)\), and define
\[
A_V^{\vee}:=(J^D(M)_V/I)^{\vee},
\qquad
B_V^{\vee}:=I\,J^D(M)_V^{\vee}.
\]
Then
\[
J^D(M)_V^{\vee}=A_V^{\vee}+B_V^{\vee}.
\]
These varieties are \(\mathbb T^D(M)\)-stable, so we define their
congruence and modular exponents analogously.

\begin{Def}
	Let
	\[
	\pi_{A_V}:\mathbb T^D(M)\to\operatorname{End}(A_V),
	\qquad
	\pi_{B_V}:\mathbb T^D(M)\to\operatorname{End}(B_V),
	\]
	and similarly define \(\pi_{A_V^\vee}\) and \(\pi_{B_V^\vee}\). Write
	\[
	\mathbb T^D(M)_{A_V}=\pi_{A_V}(\mathbb T^D(M)),
	\qquad
	\mathbb T^D(M)_{B_V}=\pi_{B_V}(\mathbb T^D(M)),
	\]
	with analogous notation for \(A_V^\vee\) and \(B_V^\vee\).
	
	\begin{enumerate}[label=(\alph*),start=\intfromalph{i}]
		\item
		\[
		R_V:=\pi_{A_V}(\ker(\pi_{B_V})),
		\qquad
		r(A_V)=\exp\!\left(\mathbb T^D(M)_{A_V}/R_V\right).
		\]
		
		\item
		\[
		S_V:=\operatorname{Ann}_{\mathbb T^D(M)_{A_V}}(A_V\cap B_V),
		\qquad
		\delta(A_V)=\exp\!\left(\mathbb T^D(M)_{A_V}/S_V\right).
		\]
		
		\item
		\[
		R_V^\vee:=\pi_{A_V^\vee}(\ker(\pi_{B_V^\vee})),
		\qquad
		r(A_V^\vee)
		=
		\exp\!\left(\mathbb T^D(M)_{A_V^\vee}/R_V^\vee\right).
		\]
		
		\item
		\[
		S_V^\vee
		:=
		\operatorname{Ann}_{\mathbb T^D(M)_{A_V^\vee}}
		(A_V^\vee\cap B_V^\vee),
		\qquad
		\delta(A_V^\vee)
		=
		\exp\!\left(\mathbb T^D(M)_{A_V^\vee}/S_V^\vee\right).
		\]
	\end{enumerate}
\end{Def}

Denote by
\[
r_{\mathfrak m}^D(M),\quad
r_{\mathfrak m}(A_V),\quad
r_{\mathfrak m}(A_V^\vee),\quad
\delta_{\mathfrak m}(A),\quad
\delta_{\mathfrak m}(A_V),\quad
\delta_{\mathfrak m}(A_V^\vee)
\]
the exponents of, respectively,
\[
\begin{gathered}
	(\mathbb T^D(M)_A/R)_{\mathfrak m},\qquad
	(\mathbb T^D(M)_{A_V}/R_V)_{\mathfrak m},\qquad
	(\mathbb T^D(M)_{A_V^\vee}/R_V^\vee)_{\mathfrak m},\\
	(\mathbb T^D(M)_A/S)_{\mathfrak m},\qquad
	(\mathbb T^D(M)_{A_V}/S_V)_{\mathfrak m},\qquad
	(\mathbb T^D(M)_{A_V^\vee}/S_V^\vee)_{\mathfrak m}.
\end{gathered}
\]

We now state the main result of this subsection.

\begin{thm}\label{thm:isogeny-exponents}
	\begin{enumerate}
		\item
		$
		r^D(M)=r(A_V)=r(A_V^\vee).
		$
		
		\item
		$
		r_{\mathfrak m}^D(M)
		=
		r_{\mathfrak m}(A_V)
		=
		r_{\mathfrak m}(A_V^\vee)
		=
		\delta_{\mathfrak m}(A_V)
		=
		\delta_{\mathfrak m}(A_V^\vee).
		$
	\end{enumerate}
\end{thm}

Part~(1) follows from tracing the definitions. For Part~(2), we use the
strategy of \cite{agashe2011modular}, together with the freeness of
$
\mathcal T_{\mathfrak m}(J^D(M)_V)=V.
$
A crucial step in loc.\ cit. uses the fact that, by strong multiplicity one
for newforms, \(A\) and \(B\) are
\(\operatorname{End}(J^D(M))_{\mathbb Q}\)-invariant. Since \(J^D(M)_V\) is
defined over a number field \(K\), the corresponding statement with the full
ring \(\operatorname{End}(J^D(M)_V)_K\) need not hold because of twists of
newforms. We therefore work with the endomorphisms commuting with the Hecke
action.

\begin{prop}\label{prop:hecke-equivariant-invariance}
	The abelian subvarieties \(A_V\) and \(B_V\) are invariant under
	$
	\operatorname{End}_{\mathbb T^D(M)}(J^D(M)_V).
	$
\end{prop}

\begin{proof}
	We prove the statement for \(A_V\); the proof for \(B_V\) is identical.
	
	Set
	$
	I_f
	=
	\ker\!\left(
	\mathbb T^D(M)\longrightarrow\operatorname{End}(A)
	\right).
	$
	By the choice of the \(f\)-isotypic factor,
	$
	A=J^D(M)[I_f]^0.
	$
	Thus \(J^D(M)[I_f]/A\) is finite; choose a positive integer \(c\) annihilating
	this quotient. Let
	$
	\pi:J^D(M)\longrightarrow J^D(M)_V
	$
	be the defining \(\mathbb T^D(M)\)-equivariant isogeny over $\bar{\mathbb{Q}}$, and choose a positive
	integer \(d\) annihilating \(\ker(\pi)\).
	We first claim that
	$
	[cd]\,J^D(M)_V[I_f]\subseteq A_V.
	$
	Indeed, let \(y\in J^D(M)_V[I_f]\) and choose \(x\in J^D(M)\) with
	\(\pi(x)=y\). For every \(t\in I_f\),
	\[
	\pi(tx)=t\pi(x)=ty=0,
	\]
	so \(tx\in\ker(\pi)\), hence \(dtx=0\). Therefore
	\(dx\in J^D(M)[I_f]\). Since \(c\) annihilates
	\(J^D(M)[I_f]/A\), we have \(cdx\in A\), and consequently
	\[
	[cd]y=\pi(cdx)\in\pi(A)=A_V.
	\]
	Now let
	$
	\phi\in
	\operatorname{End}_{\mathbb T^D(M)}(J^D(M)_V).
	$
	Since \(I_f\) annihilates \(A_V\) and \(\phi\) commutes with the Hecke action,
	for every \(t\in I_f\) and \(x\in A_V\),
	$
	t\phi(x)=\phi(tx)=0.
	$
	Thus
	$
	\phi(A_V)\subseteq J^D(M)_V[I_f],
	$
	and the preceding claim gives
	$
	[cd]\phi(A_V)\subseteq A_V.
	$
	Let
	$
	q:J^D(M)_V\longrightarrow J^D(M)_V/A_V
	$
	be the quotient map. Then
	$
	[cd]\;q\circ\phi|_{A_V}=0.
	$
	The image of \(q\circ\phi|_{A_V}\) is a connected abelian subvariety of
	\(J^D(M)_V/A_V\), while the kernel of multiplication by \(cd\) is finite.
	Hence this image is trivial, so \(\phi(A_V)\subseteq A_V\).
	
	For \(B_V\), the same argument applies with
	\(\operatorname{Ann}_{\mathbb T^D(M)}(I_f)\) in place of \(I_f\), using
	$
	B
	=
	J^D(M)
	[\operatorname{Ann}_{\mathbb T^D(M)}(I_f)]^0.
	$
	This proves the proposition.
\end{proof}

\begin{proof}[Proof of Theorem~\ref{thm:isogeny-exponents}]
	\begin{enumerate}
		\item
		We first show \(r^D(M)=r(A_V)\); the equality
		\(r^D(M)=r(A_V^\vee)\) follows similarly. The argument is that of Theorem~A
		in \cite{hamidi2026}, with a minor modification.
		
		By construction, there is a \(\mathbb T^D(M)\)-equivariant isogeny
		$
		\phi:J^D(M)\longrightarrow J^D(M)_V.
		$
		If \(t\in\ker(\pi_B)\) and \(y\in B_V\), choose \(x\in B\) with
		\(y=\phi(x)\). Then
		\[
		ty=t\phi(x)=\phi(tx)=0,
		\]
		so \(t\in\ker(\pi_{B_V})\). Thus
		$
		\ker(\pi_B)\hookrightarrow\ker(\pi_{B_V}),
		$
		and similarly
		$
		\ker(\pi_A)\hookrightarrow\ker(\pi_{A_V}).
		$
		Hence
		\[
		R^D(M)=\pi_A(\ker(\pi_B))
		\hookrightarrow
		\pi_{A_V}(\ker(\pi_{B_V}))=R_V.
		\]
		The identity on \(\mathbb T^D(M)\) therefore induces a surjection
		\[
		\mathbb T^D(M)/R^D(M)
		\twoheadrightarrow
		\mathbb T^D(M)/R_V,
		\]
		and consequently \(r^D(M)\geq r(A_V)\).
		
		For the reverse inequality, consider a dual isogeny
		\[
		\widehat\phi:J^D(M)_V\longrightarrow J^D(M)
		\]
		with \(\widehat\phi\circ\phi=[n]\). Since multiplication by \(n\) commutes
		with the Hecke action and
		\(\widehat\phi(A_V)=A\), \(\widehat\phi(B_V)=B\), the same argument applied
		to \(\widehat\phi\) gives \(R_V\hookrightarrow R^D(M)\), and hence
		\(r(A_V)\geq r^D(M)\). Therefore
		\[
		r^D(M)=r(A_V).
		\]
		
		\item
		The equality
		\[
		r_{\mathfrak m}^D(M)
		=
		r_{\mathfrak m}(A_V)
		=
		r_{\mathfrak m}(A_V^\vee)
		\]
		follows from Part~(1) after localization. It remains to prove
		\[
		r_{\mathfrak m}(A_V)=\delta_{\mathfrak m}(A_V),
		\qquad
		r_{\mathfrak m}(A_V^\vee)=\delta_{\mathfrak m}(A_V^\vee).
		\]
		
		By Proposition~\ref{prop:hecke-equivariant-invariance}, restriction gives a
		natural inclusion
		\[
		\operatorname{End}_{\mathbb T^D(M)}(J^D(M)_V)
		\hookrightarrow
		\operatorname{End}_{\mathbb T^D(M)_{A_V}}(A_V)
		\oplus
		\operatorname{End}_{\mathbb T^D(M)_{B_V}}(B_V).
		\]
		Let \(\mathbb T^D(M)'\) be the saturation of \(\mathbb T^D(M)\) inside
		\(\operatorname{End}_{\mathbb T^D(M)}(J^D(M)_V)\):
		\[
		\mathbb T^D(M)'
		=
		\left\{
		\phi\in
		\operatorname{End}_{\mathbb T^D(M)}(J^D(M)_V)
		:
		n\phi\in\mathbb T^D(M)
		\text{ for some }n\in\mathbb N
		\right\}.
		\]
		Inside the preceding direct sum,
		\[
		S_V=\mathbb T^D(M)'\cap\mathbb T^D(M)_{A_V},
		\qquad
		R_V=\mathbb T^D(M)\cap\mathbb T^D(M)_{A_V}.
		\]
		Therefore
		\[
		\frac{S_V}{R_V}
		\cong
		\frac{
			\mathbb T^D(M)'\cap\mathbb T^D(M)_{A_V}+\mathbb T^D(M)
		}{
			\mathbb T^D(M)
		}
		\hookrightarrow
		\frac{\mathbb T^D(M)'}{\mathbb T^D(M)}.
		\]
		Since
		\[
		\mathcal T_{\mathfrak m}(J^D(M)_V)
		=
		V
		\cong
		T^{\,2},
		\]
		the argument of Lemma~5.8 in \cite{agashe2011modular} implies
		$
		\mathbb T^D(M)'_{\mathfrak m}
		=
		\mathbb T^D(M)_{\mathfrak m}.
		$
		Thus
		$
		(S_V)_{\mathfrak m}=(R_V)_{\mathfrak m},
		$
		and therefore
		\[
		r_{\mathfrak m}(A_V)=\delta_{\mathfrak m}(A_V).
		\]
		
		For \(A_V^\vee\), by Proposition~\ref{prop:dual-tate-module},
		\[
		\widetilde{\mathcal T}_{\mathfrak m}(J^D(M)_V^\vee)
		\cong
		\mathcal T_{\mathfrak m}(J^D(M)_V)
		\cong
		V
		\cong
		T^{\,2}.
		\]
		Applying the same argument gives
		$
		r_{\mathfrak m}(A_V^\vee)
		=
		\delta_{\mathfrak m}(A_V^\vee).
		$
	\end{enumerate}
\end{proof}

\subsection{Proof of Theorem A}

\begin{proof}[Proof of Theorem A]
	Let \(V\) be a lattice. Recall that $	\widetilde{\mathcal T}_{\mathfrak m}(J^D(M)_V^\vee)$ denotes the covariant $\mathfrak{m}$-adic Tate module of $J^D(M)_V^\vee$. By construction and
	Proposition~\ref{prop:dual-tate-module},
	\[
	\widetilde{\mathcal T}_{\mathfrak m}(J^D(M)_V^\vee)
	\cong
	\mathcal T_{\mathfrak m}(J^D(M)_V)
	\cong V.
	\]
	By Theorem~\ref{thm:isogeny-exponents},
	$
	r_{\mathfrak m}(A_V^\vee)
	=
	\delta_{\mathfrak m}(A_V^\vee).
	$
	
	We compare these quantities through their homological realizations. By
	\cite[Proposition~3.2]{agashe2005visible},
	\[
	A\cap B
	\cong
	\frac{\mathrm H_1(J^D(M),\mathbb Z)}
	{\mathrm H_1(J^D(M),\mathbb Z)[I_f]
		+
		\mathrm H_1(J^D(M),\mathbb Z)
		[\operatorname{Ann}_{\mathbb T^D(M)}(I_f)]},
	\]
	and
	\[
	A_V^\vee\cap B_V^\vee
	\cong
	\frac{\mathrm H_1(J^D(M)_V^\vee,\mathbb Z)}
	{\mathrm H_1(J^D(M)_V^\vee,\mathbb Z)[I_f]
		+
		\mathrm H_1(J^D(M)_V^\vee,\mathbb Z)
		[\operatorname{Ann}_{\mathbb T^D(M)}(I_f)]}.
	\]
	By the standard duality for abelian varieties,
	\[
	\mathrm H_1(J^D(M)_V^\vee,\mathbb Z)
	\cong
	\mathrm H^1(J^D(M)_V,\mathbb Z),
	\]
	so localizing at \(\mathfrak m\) and identifying the resulting modules with
	the corresponding contravariant \(\mathfrak m\)-adic Tate modules gives
	\[
	(A\cap B)_{\mathfrak m}
	\cong
	\frac{\mathcal T_{\mathfrak m}(J^D(M))}
	{\mathcal T_{\mathfrak m}(J^D(M))[I_f]
		+
		\mathcal T_{\mathfrak m}(J^D(M))
		[\operatorname{Ann}_{\mathbb T^D(M)}(I_f)]}
	\]
	and
	\[
	(A_V^\vee\cap B_V^\vee)_{\mathfrak m}
	\cong
	\frac{\mathcal T_{\mathfrak m}(J^D(M)_V)}
	{\mathcal T_{\mathfrak m}(J^D(M)_V)[I_f]
		+
		\mathcal T_{\mathfrak m}(J^D(M)_V)
		[\operatorname{Ann}_{\mathbb T^D(M)}(I_f)]}.
	\]
	
	The inclusion
	\[
	i:
	\mathcal T_{\mathfrak m}(J^D(M)_V)
	\hookrightarrow
	\mathcal T_{\mathfrak m}(J^D(M))
	\]
	induces a natural map
	\[
	\widetilde i:
	\frac{\mathcal T_{\mathfrak m}(J^D(M)_V)}
	{\mathcal T_{\mathfrak m}(J^D(M)_V)[I_f]
		+
		\mathcal T_{\mathfrak m}(J^D(M)_V)
		[\operatorname{Ann}_{\mathbb T^D(M)}(I_f)]}
	\longrightarrow
	\frac{\mathcal T_{\mathfrak m}(J^D(M))}
	{\mathcal T_{\mathfrak m}(J^D(M))[I_f]
		+
		\mathcal T_{\mathfrak m}(J^D(M))
		[\operatorname{Ann}_{\mathbb T^D(M)}(I_f)]}.
	\]
	
	Let \([x]\in\ker(\widetilde i)\). Then
	$
	i(x)=a+b
	$
	for some
	\[
	a\in\mathcal T_{\mathfrak m}(J^D(M))[I_f],
	\qquad
	b\in
	\mathcal T_{\mathfrak m}(J^D(M))
	[\operatorname{Ann}_{\mathbb T^D(M)}(I_f)].
	\]
	Put
	\[
	e_V
	=
	\exp\!\left(
	\frac{\mathcal T_{\mathfrak m}(J^D(M))[I_f]+V}{V}
	\right).
	\]
	Then \(e_Va\in V\). Choose \(a'\in V\) with \(i(a')=e_Va\). Since \(i\)
	is injective and \(T\)-equivariant,
	\(a'\in V[I_f]\). Moreover,
	\[
	e_Vb
	=
	e_Vi(x)-e_Va
	=
	i(e_Vx-a').
	\]
	Because \(b\) is annihilated by
	\(\operatorname{Ann}_{\mathbb T^D(M)}(I_f)\), injectivity and equivariance of
	\(i\) imply
	\[
	b':=e_Vx-a'
	\in
	V[\operatorname{Ann}_{\mathbb T^D(M)}(I_f)].
	\]
	Therefore
	\[
	e_Vx
	=
	a'+b'
	\in
	\mathcal T_{\mathfrak m}(J^D(M)_V)[I_f]
	+
	\mathcal T_{\mathfrak m}(J^D(M)_V)
	[\operatorname{Ann}_{\mathbb T^D(M)}(I_f)].
	\]
	Thus \(\ker(\widetilde i)\) has exponent at most \(e_V\), and hence
	\[
	r_{\mathfrak m}^{D}(M)
	=
	\delta_{\mathfrak m}(A_V^\vee)
	\leq
	\delta_{\mathfrak m}^{D}(M)\,e_V.
	\]
	Since this holds for every lattice \(V\),
	\[
	r_{\mathfrak m}^{D}(M)
	\leq
	\delta_{\mathfrak m}^{D}(M)
	\min_V
	\exp\!\left(
	\frac{\mathcal T_{\mathfrak m}(J^D(M))[I_f]+V}{V}
	\right).
	\]
	For every \(V\),
	\[
	\frac{\mathcal T_{\mathfrak m}(J^D(M))[I_f]+V}{V}
	\subseteq
	\frac{\mathcal T_{\mathfrak m}(J^D(M))}{V},
	\]
	so
	\[
	\min_V
	\exp\!\left(
	\frac{\mathcal T_{\mathfrak m}(J^D(M))[I_f]+V}{V}
	\right)
	\leq
	\delta_{\mathrm{def}}(\mathfrak m).
	\]
	Therefore
	$
	r_{\mathfrak m}^{D}(M)
	\leq
	\delta_{\mathfrak m}^{D}(M)\,
	\delta_{\mathrm{def}}(\mathfrak m),
	$
	which proves Theorem~A.
\end{proof}

\subsection{Second Proof}

We give a second proof of Theorem~A using the saturation of the Hecke algebra
in the endomorphism algebra.

\begin{proof}[Second proof of Theorem A]
	Let \(V\) be a lattice with associated abelian variety \(J^D(M)_V\), and let
	$
	\pi:J^D(M)\longrightarrow J^D(M)_V
	$
	be the corresponding \(\mathbb T^D(M)\)-equivariant isogeny with kernel
	\(G\). On contravariant \(\mathfrak m\)-adic Tate modules, \(\pi\) induces
	the inclusion
	$
	V
	=
	\mathcal T_{\mathfrak m}(J^D(M)_V)
	\hookrightarrow
	\mathcal T_{\mathfrak m}(J^D(M)),
	$
	whose cokernel is
	\[
	\mathcal T_{\mathfrak m}(J^D(M))/V.
	\]
	The exponent of \(G\) equals the exponent of this quotient. Denote this common
	exponent by \(e\).
	Since \(G\) is annihilated by \(e\), multiplication by \(e\) on \(J^D(M)\)
	factors through \(\pi\). Thus there is a
	\(\mathbb T^D(M)\)-equivariant isogeny
	$
	\pi^\vee:J^D(M)_V\longrightarrow J^D(M)
	$
	such that
	$
	\pi^\vee\circ\pi=[e],
	~
	\pi\circ\pi^\vee=[e].
	$
	Let \(\mathbb T^D(M)'\) and \(\mathbb T^D(M)'_V\) denote the saturations of
	\(\mathbb T^D(M)\) in \(\operatorname{End}(J^D(M))\) and
	\(\operatorname{End}(J^D(M)_V)\), respectively. By the proof of
	Theorem~\ref{thm:isogeny-exponents}, since
	$
	\mathcal T_{\mathfrak m}(J^D(M)_V)
	\cong V \cong T^{\,2},
	$
	we have
	$
	\widehat{\mathbb T^D(M)'_{V,\mathfrak m}}=T.
	$
    Consequently,
	$
	\mathbb T^D(M)'_{V,\mathfrak m}
	=
	\mathbb T^D(M)_{\mathfrak m}.
	$
	Consider
	\[
	\Phi:\operatorname{End}(J^D(M))
	\longrightarrow
	\operatorname{End}(J^D(M)_V),
	\qquad
	\Phi(\phi)=\pi\circ\phi\circ\pi^\vee.
	\]
	Since \(\pi^\vee\circ\pi=[e]\), the map \(\Phi\) is injective. Indeed, if
	\(\Phi(\phi)=0\), then composing on the left by \(\pi^\vee\) and on the
	right by \(\pi\) gives
	\[
	[e^2]\circ\phi=0.
	\]
	The endomorphism ring of an abelian variety is torsion-free, so \(\phi=0\).
	If \(\phi\in\mathbb T^D(M)'\), choose a positive integer \(n\) such that
	\(n\phi\in\mathbb T^D(M)\). By Hecke equivariance,
	\[
	n\Phi(\phi)
	=
	\pi\circ(n\phi)\circ\pi^\vee
	=
	e(n\phi)
	\in\mathbb T^D(M),
	\]
	where the right-hand side acts on \(J^D(M)_V\). Thus
	\[
	\Phi(\mathbb T^D(M)')\subseteq\mathbb T^D(M)'_V.
	\]
	After localization, this gives
	\[
	\Phi:
	\mathbb T^D(M)'_{\mathfrak m}
	\longrightarrow
	\mathbb T^D(M)'_{V,\mathfrak m}.
	\]
	Let \(\phi\in\mathbb T^D(M)'_{\mathfrak m}\). Since
	$
	\mathbb T^D(M)'_{V,\mathfrak m}
	=
	\mathbb T^D(M)_{\mathfrak m},
	$
	there exists \(t\in\mathbb T^D(M)_{\mathfrak m}\) such that
	\[
	\pi\circ\phi\circ\pi^\vee=t.
	\]
	Composing on the left by \(\pi^\vee\) and on the right by \(\pi\), and using
	Hecke equivariance, gives
	$
	e^2\phi=et.
	$
	Since the localized endomorphism module is torsion-free over \(\mathbb Z\),
	we may cancel \(e\) to obtain
	$
	e\phi=t\in\mathbb T^D(M)_{\mathfrak m}.
	$
	Consequently,
	$
	e\,\frac{\mathbb T^D(M)'_{\mathfrak m}}
	{\mathbb T^D(M)_{\mathfrak m}}=0.
	$
	
	It follows that
	\[
	\exp\!\left(
	\frac{\mathbb T^D(M)'_{\mathfrak m}}
	{\mathbb T^D(M)_{\mathfrak m}}
	\right)
	\leq e
	=
	\exp\!\left(
	\mathcal T_{\mathfrak m}(J^D(M))/V
	\right).
	\]
	Minimizing over all lattices \(V\) gives
	\[
	\exp\!\left(
	\frac{\mathbb T^D(M)'_{\mathfrak m}}
	{\mathbb T^D(M)_{\mathfrak m}}
	\right)
	\leq
	\delta_{\mathrm{def}}(\mathfrak m).
	\]
	By Proposition~3.4 of \cite{hamidi2026}, there is an inclusion
	$
	\frac{S^D(M)_{\mathfrak m}}{R^D(M)_{\mathfrak m}}
	\hookrightarrow
	\frac{\mathbb T^D(M)'_{\mathfrak m}}
	{\mathbb T^D(M)_{\mathfrak m}}.
	$
	Hence
	\[
	\exp\!\left(
	S^D(M)_{\mathfrak m}/R^D(M)_{\mathfrak m}
	\right)
	\leq
	\delta_{\mathrm{def}}(\mathfrak m).
	\]
	Equivalently,
	\[
	r_{\mathfrak m}^{D}(M)
	\leq
	\delta_{\mathfrak m}^{D}(M)\,
	\delta_{\mathrm{def}}(\mathfrak m),
	\]
	which completes the proof.
\end{proof}

%% file: Gorenstein.tex
\section{Gorensteinness and Multiplicity One}

We exhibit an explicit example showing that the Gorenstein property does not imply multiplicity one. We also compute the multiplicity-one defect in this example and show that
\[
\delta_{\mathrm{def}}(\mathfrak m)
=
\frac{\text{congruence exponent}}{\text{modular exponent}}
=3.
\]
Thus the lower bound for the multiplicity-one defect given by Theorem A is sharp. 

\begin{exa}\label{ex:126-gorenstein}
	Let $E=126A1$ in Cremona notation \cite{cremona1997algorithms}. In the
	LMFDB \cite{lmfdb} this is the $\Gamma_0(126)$-optimal curve $126.b5$
	in the isogeny class $126.b$, and the corresponding normalized newform is
	the rational newform
	$f=f_E=126.2.a.b$.
	The modular degree and congruence number of $E$ are
	$
	m_E=8,
	~
	r_E=24.
	$
	Thus
	$
	3\mid r_E,
	~
	3\nmid m_E,
	$
	and hence multiplicity one fails at the maximal ideal $\mathfrak m$ of
	residue characteristic $3$ associated to $f$ by
	\cite{agashe2011modular}.
	We show nevertheless that the localized Hecke algebra is Gorenstein.
	More precisely, if $\mathbb T$ denotes the full Hecke algebra acting on
	$S_2(\Gamma_0(126))$ and
	$
	\lambda_f:\mathbb T\longrightarrow\mathbb Z
	$
	is the system of Hecke eigenvalues attached to $f$, and
	$
	\mathfrak m=(3,\ker\lambda_f),
	$
	then the $\mathfrak m$-adic completion satisfies
	$
	\widehat{\mathbb T_{\mathfrak m}}
	\simeq
	\mathbb Z_3\times_{\mathbb F_3}\mathcal O,
	~\text{and}~
	\mathcal O=\mathbb Z_3[\sqrt3],
	$
	where the maps to $\mathbb F_3$ are reduction modulo $3$ on
	$\mathbb Z_3$ and reduction modulo the maximal ideal
	$
	\lambda=(\sqrt3)
	$
	on $\mathcal O$. Equivalently,
	\[
	\widehat{\mathbb T_{\mathfrak m}}
	\simeq
	\left\{
	(a,b)\in\mathbb Z_3\times\mathcal O:
	b-a\in\sqrt3\,\mathcal O
	\right\}.
	\]
	We first determine the eigenspaces which survive after
	localization at $\mathfrak m$. Since
	$
	126=2\cdot3^2\cdot7,
	$
	the relevant newform levels dividing $126$ are
	$
	14,21,42,63,\text{ and }126,
	$
	and the newform orbits are
	$
	14.2.a.a,
	21.2.a.a,
	42.2.a.a,
	63.2.a.a,
	63.2.a.b,
	126.2.a.a,
	\text{ and }
	126.2.a.b=f,
	$
	as in \cite{lmfdb}.
	The $U_7$-eigenvalues are
	\[
	\begin{array}{c|c|c}
		\text{level} & \text{newform orbit} & U_7\text{-eigenvalue }(a_7) \\ \hline
		14  & 14.2.a.a  & 1\\
		21  & 21.2.a.a  & -1\\
		42  & 42.2.a.a  & -1\\
		63  & 63.2.a.a  & -1\\
		63  & 63.2.a.b  & 1\\
		126 & 126.2.a.a & -1\\
		126 & 126.2.a.b & 1
	\end{array}
	\]
	By $a_7$, we mean the $7$th Fourier coefficient of the corresponding normalized newform in each orbit.
	Since $a_7(f)=1$, every eigenspace on which $U_7$ has eigenvalue $-1$ is eliminated after
	localization. Thus the only possibilities remaining after
	the $U_7$ comparison are $f$, the level-$14$ old eigenspace, and the eigenspace
	arising from $63.2.a.b$.
	
	The level-$14$ contribution is eliminated by $U_2$. Indeed, if
	$
	g=14.2.a.a,
	$
	then $a_2(g)=-1$, and because the passage from level $14$ to level $126$
	is entirely in the $3$-direction, $U_2$ acts by $-1$ on the entire
	level-$14$ oldspace.
	It remains to analyze
	$h=63.2.a.b$. Let $\beta=\sqrt{3}$ and
	write
	\[
	K=\mathbb Q(\beta),
	\qquad
	\beta^2=3,
	\qquad
	\mathcal O=\mathbb Z_3[\beta],
	\qquad
	\lambda=(\beta).
	\]
	The $q$-expansion of $h$ begins
	\[
	h(q)
	=
	q+\beta q^2+q^4-2\beta q^5+q^7-\beta q^8
	-6q^{10}+2\beta q^{11}+2q^{13}+\cdots.
	\]
	A comparison of coefficients through the Sturm bound shows that the
	Hecke eigenvalues of $h$ away from $126$ reduce modulo $\lambda$ to the
	corresponding eigenvalues of $f$.
	
	Since $126=2\cdot63$, the form $h$ contributes a $2$-old space at level
	$126$. On this oldspace the characteristic polynomial of $U_2$ is
	$
	X^2-\beta X+2.
	$
	Modulo $\lambda=(\beta)$ this becomes
	\[
	X^2+2=X^2-1=(X-1)(X+1).
	\]
	Because $U_2(f)=1$, localization at $\mathfrak m$ selects the root
	$\alpha$ satisfying
	$
	\alpha\equiv1\pmod\lambda;
	$
	the other factor $X+1$ becomes a unit locally at $\mathfrak m$.
	
	The root $\alpha$ lies in $K_\lambda=\mathbb Q_3(\sqrt3)$. We now show
	that $\alpha-1$ is a uniformizer of $\mathcal O$. Put
	$
	\pi=\alpha-1.
	$
	Since $\alpha\equiv1\pmod\lambda$, we have
	$
	v_\lambda(\pi)>0.
	$
	Writing the equation
	$
	\alpha^2-\beta\alpha+2=0
	$
	in terms of $\pi$ gives
	\[
	(1+\pi)^2-\beta(1+\pi)+2=0,
	\]
	hence
	\[
	\pi^2+(2-\beta)\pi+(3-\beta)=0.
	\]
	Equivalently,
	\[
	\pi(\pi+2-\beta)=-(3-\beta).
	\]
	Now $2-\beta$ is a unit in $\mathcal O$, while
	$
	v_\lambda(3-\beta)=1,
	$
	because $v_\lambda(\beta)=1$ and $v_\lambda(3)=2$. Since
	$v_\lambda(\pi)>0$, the element $\pi+2-\beta$ is a unit. Therefore
	\[
	v_\lambda(\pi)
	=
	v_\lambda(3-\beta)
	=
	1.
	\]
	Thus $\pi=\alpha-1$ is a uniformizer of $\mathcal O$.
	Since $\alpha\in\mathcal O=\mathbb Z_3[\beta]$, write
	$
	\alpha=a+b\beta,
	~
	a,b\in\mathbb Z_3.
	$
As \(\alpha\equiv1\pmod\lambda\), we have \(v_\lambda(a-1)\ge1\). Since \(a-1\in\mathbb Z_3\) and \(v_\lambda\) takes even values on \(\mathbb Q_3^\times\), it follows that
$
v_\lambda(a-1)\ge2$. On the other hand,
	$
	\alpha-1=(a-1)+b\beta
	$
	has $\lambda$-adic valuation $1$ and thus $b\in\mathbb Z_3^\times$.
	Therefore
	$
	\beta=b^{-1}(\alpha-a)\in\mathbb Z_3[\alpha],
	$
	and hence
	\[
	\mathbb Z_3[\alpha-1]
	=
	\mathbb Z_3[\alpha]
	=
	\mathbb Z_3[\beta]
	=
	\mathcal O.
	\]
	We therefore obtain
	\[
	\widehat{\mathbb T_{\mathfrak m}}
	\otimes_{\mathbb Z_3}\mathbb Q_3
	\simeq
	\mathbb Q_3\times\mathbb Q_3(\sqrt3),
	\]
	where the first factor is the $f$-branch and the second is the surviving $2$-old eigenspace coming from $63.2.a.b$.
	
	We now determine the integral gluing. Under the two eigenvalue maps, every
	element of $\widehat{\mathbb T_{\mathfrak m}}$ has the same residue on the
	two branches, because both branches lie over $\mathfrak m$. Hence
	\[
	\widehat{\mathbb T_{\mathfrak m}}
	\subseteq
	\left\{
	(a,b)\in\mathbb Z_3\times\mathcal O:
	b-a\in\lambda\mathcal O
	\right\}.
	\]
	On the other hand,
	\[
	U_2-1\longmapsto(0,\alpha-1)=(0,\pi).
	\]
	Since
	$
	\mathcal O=\mathbb Z_3[\pi]
	~\text{and}~
	\pi\mathcal O=\lambda\mathcal O,
	$
	the subring generated by the scalar diagonal $\mathbb Z_3$ and $U_2-1$
	contains
	$
	0\times\lambda\mathcal O.
	$
	Indeed, if $c\in\mathcal O$, write $c=g(\pi)$ with
	$g(X)\in\mathbb Z_3[X]$. Then, with $t=U_2-1$,
	\[
	tg(t)
	\longmapsto
	(0,\pi)\bigl(g(0),g(\pi)\bigr)
	=
	(0,\pi c).
	\]
	Hence every element of
	$0\times\pi\mathcal O=0\times\lambda\mathcal O$ lies in the image of
	$\widehat{\mathbb T_{\mathfrak m}}$. It follows that every pair $(a,b)$
	satisfying $b-a\in\lambda\mathcal O$ lies in the image, since
	$
	(a,b)=(a,a)+(0,b-a).
	$
	We conclude that
	\[
	\widehat{\mathbb T_{\mathfrak m}}
	\simeq
	\mathbb Z_3\times_{\mathbb F_3}\mathcal O.
	\]
	The preceding computation is reminiscent of Kilford's method for his
	explicit non-Gorenstein Hecke algebras \cite{kilford2002some}. There, after
	identifying the characteristic-zero eigenforms lying above a fixed
	residual system of Hecke eigenvalues, the completed local Hecke algebra
	is realized explicitly as an integral subring of the product of the
	corresponding $p$-adic coefficient rings. Our computation follows the
	same general strategy: we first determine the characteristic-zero
	eigenspaces which survive localization at $\mathfrak m$, obtaining
	$
	\mathbb Q_3\times\mathbb Q_3(\sqrt3),
	$
	and then determine the integral gluing of the two branches from their
	Hecke eigenvalues. The outcome, however, is opposite to Kilford's
	examples. Whereas his explicit integral orders are non-Gorenstein, here
	the gluing is
	$
	\mathbb Z_3\times_{\mathbb F_3}\mathbb Z_3[\sqrt3],
	$
	which is Gorenstein. Thus the same computational perspective that detects
	failure of the Gorenstein property in Kilford's examples produces here a
	Gorenstein local algebra despite the failure of multiplicity one.
	Since
	$\beta\in\lambda\mathcal O$, the preceding paragraph shows that
	$
	x=(0,\beta)
	$
	belongs to $\widehat{\mathbb T_{\mathfrak m}}$. Then
	$
	x^2=(0,3),
	x^3=3x,
	$
	and $1,x,x^2$ form a $\mathbb Z_3$-basis. Thus
	\[
	\widehat{\mathbb T_{\mathfrak m}}
	\simeq
	\mathbb Z_3[X]/(X^3-3X)
	=
	\mathbb Z_3[X]/\bigl(X(X^2-3)\bigr).
	\]
	Hence $\widehat{\mathbb T_{\mathfrak m}}$ is 
	complete intersection and, in particular, Gorenstein.	Since a Noetherian local ring is Gorenstein if and only if
	its completion is Gorenstein, we conclude that
	$\mathbb T_{\mathfrak m}$ is Gorenstein.

	We have therefore exhibited a maximal ideal for which
	$\mathbb T_{\mathfrak m}$ is Gorenstein while multiplicity one fails.
	We finish the example by computing $\delta_{\mathrm{def}}(\mathfrak{m})=3$.
	Put \(J=J_0(126)\), and let \(A\subset J\)
	be the elliptic curve associated with
	\(f\). It is the dual of the optimal elliptic curve $126.b5$. Let \(B\) be its Hecke-stable
	complementary abelian subvariety.
	The modular degree of \(A\) is \(8\), i.e. the intersection
	\(A\cap B\) has exponent \(8\).
	Consequently, the addition isogeny
	$
	A\times B\longrightarrow J
	$
	has kernel of order prime to \(3\).
	It therefore induces a Hecke-equivariant
	isomorphism on the covariant $\mathfrak{m}$-adic  Tate modules
	\[
   L=T_\mathfrak{m}(J)\simeq 	T_\mathfrak{m}(A)\oplus T_\mathfrak{m}(B).
	\]
	By the preceding computation of the local
	Hecke algebra,
	$
	\widehat{\mathbb T_{\mathfrak m}}\simeq
	\mathbb Z_3\times_{\mathbb F_3}\mathcal O.
	$
	The elliptic factor contributes
	$
    T_{\mathfrak m}(A)\simeq\mathbb Z_3^2.
	$
	The only other characteristic-zero Hecke
	component surviving at \(\mathfrak m\)
	is the component arising from the
	level-\(63\) newform considered above.
	The image of the local Hecke algebra on
	this component is \(\mathcal O\).
	Hence \(T_{\mathfrak m}(B)\) is a finite
	torsion-free \(\mathcal O\)-module
	of rank \(2\) over  \(\mathcal O\).
	We therefore obtain an isomorphism of
	\(\widehat{\mathbb T_{\mathfrak m}}'\)-modules
	$
	L\simeq
	\mathbb Z_3^2\oplus\mathcal O^2
	\simeq 	\widehat{\mathbb T_{\mathfrak m}}'^{\,2},
	$
	where $\widehat{\mathbb T_{\mathfrak m}}'=\mathbb Z_3\times\mathcal O$.
	Under this identification, consider the
	natural inclusion
	$
	V=	\widehat{\mathbb T_{\mathfrak m}}^2\subseteq \widehat{\mathbb T_{\mathfrak m}}'^{\,2}=L.
	$
	Then \(V\) is a free \(\widehat{\mathbb T_{\mathfrak m}}\)-module of
	rank \(2\) and has finite index in \(L\).
	By the fiber-product description,
	\[
	\widehat{\mathbb T_{\mathfrak m}}=
	\left\{
	(a,b)\in\mathbf Z_3\times\mathcal O:
	b-a \in{\sqrt{3}\mathcal O}
	\right\},
	\]
	we conclude that
	$
	\widehat{\mathbb T_{\mathfrak m}}'/\widehat{\mathbb T_{\mathfrak m}}\simeq\mathbb F_3
	$
	as abelian groups. Consequently,
	\[
	L/V
	\simeq(\widehat{\mathbb T_{\mathfrak m}}'/\widehat{\mathbb T_{\mathfrak m}})^2
	\simeq(\mathbb Z/3\mathbb Z)^2.
	\]
	In particular,
	$
	\exp(L/V)=3,
	$
	and therefore
	\[
	\delta_{\mathrm{def}}(\mathfrak m)\leq 3.
	\]
	Since $m_E=8$ and $r_E=24$, by Theorem A, $	\delta_{\mathrm{def}}(\mathfrak m)\geq 3$, and hence $	\delta_{\mathrm{def}}(\mathfrak m)=3$.
\end{exa}

\medskip
\textbf{Acknowledgments.}
I am very grateful to Patrick Allen for his guidance, support, and many helpful
discussions throughout this work. I used the free version of ChatGPT
(GPT-5.6) during the preparation of the manuscript solely for language and
stylistic editing. All mathematical ideas, arguments, and results are entirely
my own, and I take full responsibility for the mathematical content.

%% file: Refe.bib
@incollection{agashe2011modular,
	title={The modular degree, congruence primes, and multiplicity one},
	author={Agashe, Amod and Ribet, Kenneth A and Stein, William A},
	booktitle={Number Theory, Analysis and Geometry: In Memory of Serge Lang},
	pages={19--49},
	year={2011},
	publisher={Springer}
}

@article{agashe2005visible,
	title={Visible evidence for the Birch and Swinnerton-Dyer conjecture for modular abelian varieties of analytic rank zero},
	author={Agashe, Amod and Stein, William},
	journal={Mathematics of Computation},
	volume={74},
	number={249},
	pages={455--484},
	year={2005}
}

@article{helm2007maps,
	title={On maps between modular Jacobians and Jacobians of Shimura curves},
	author={Helm, David},
	journal={Israel Journal of Mathematics},
	volume={160},
	number={1},
	pages={61--117},
	year={2007},
	publisher={Springer}
}

@misc{conrad2015abelian,
	author       = {Conrad, Brian},
	title        = {Abelian Varieties},
	note         = {Lecture notes by Tony Feng, Stanford University},
	year         = {2015}
}

@article{hamidi2026,
title={On the Comparison of the D-new Modular Degree and the Shimura Degree of Modular Abelian Varieties},
author={Hamidi, Mohammad Masih},
journal={arXiv preprint arXiv:2607.23244},
year={2026}
}

@misc{lmfdb,
author       = {The {LMFDB Collaboration}},
title        = {The {L}-functions and modular forms database},
howpublished = {\url{https://www.lmfdb.org}},
year         = {2026},
note         = {[Online; accessed 29 August 2026]}
}

@book{cremona1997algorithms,
author    = {Cremona, John E.},
title     = {Algorithms for Modular Elliptic Curves},
edition   = {2},
publisher = {Cambridge University Press},
year      = {1997}
}

@article{kilford2002some,
	title={Some non-Gorenstein Hecke algebras attached to spaces of modular forms},
	author={Kilford, L. J. P.},
	journal={Journal of Number Theory},
	volume={97},
	number={1},
	pages={157--164},
	year={2002},
	publisher={Elsevier}
}
